\documentclass[a4paper]{amsart}
\usepackage[textwidth=15cm,top=3cm, bottom=3cm, hcentering]{geometry}
\usepackage[colorlinks=true,linkcolor=red,citecolor=blue]{hyperref}

\usepackage{amsmath,amscd,amsthm,amssymb}
\usepackage{color}
\usepackage[utf8]{inputenc}
\normalfont
\usepackage[T1]{fontenc}
\usepackage{tikz-cd}
\usepackage{tikz}
\usepackage{enumerate}
 \usepackage[all]{xy}

\newtheorem{theo}{Theorem}[section]
\newtheorem{lemm}[theo]{Lemma}
\newtheorem{prop}[theo]{Proposition}

\theoremstyle{definition}
\newtheorem{defi}[theo]{Definition}
\newtheorem{rema}[theo]{Remark}

\newcommand{\CC}{\mathbb{C}}

\newcommand{\QQ}{\mathbb{Q}}
\newcommand{\RR}{\mathbb{R}}

\newcommand{\ZZ}{\mathbb{Z}}
\newcommand{\Aa}{\mathcal{A}}
\newcommand{\Bb}{\mathcal{B}}
\newcommand{\Cc}{\mathcal{C}}
\newcommand{\Dd}{\mathcal{D}}

\newcommand{\Ll}{\mathcal{L}}
\newcommand{\Mm}{\mathcal{M}}

\newcommand{\Ker}{\mathrm{Ker}\,}

\newcommand{\ov}{\overline}

\newcommand{\kk}{\mathbf{k}}

\newcommand{\del}{\partial}
\newcommand{\delb}{{\overline\partial}}

\newcommand{\lra}{\longrightarrow}

\newcommand{\dR}{\mathrm{dR}}

\author{Joana Cirici}

\address[J. Cirici]{Departament de Matemàtiques i Informàtica, Universitat de Barcelona | Centre de Recerca Matemàtica.
 Barcelona, Spain}
\email{jcirici@ub.edu}

\title{On the real homotopy type
of compact complex surfaces}

\thanks{Funding:
Govern de Catalunya (2021-SGR-00697 and Serra H\'{u}nter Program); Spanish State Research Agency (CEX2020-001084-M, PID2020-117971GB-C22, PID2024-155646NB-I00 and EUR2023-143450).
}
\begin{document}

\maketitle

\begin{abstract}
We show that on a compact complex surface all Massey products of cohomology classes in degree one vanish beyond length three. Dually, the real Malcev completion of the fundamental group is homogeneously presented by quadratic and cubic relations.
In the non-Kähler case, we give an explicit presentation, which is determined by the first Betti number of the surface.
\end{abstract}

\section{Introduction}
A compact complex surface is Kähler if and only if its first Betti number \(b^1\) is even and, in this case, the manifold is formal in the sense of rational homotopy theory \cite{DGMS}. In particular, Massey products vanish.
Dually, one can describe the real Malcev completion of the fundamental group of a compact Kähler manifold as a  quadratically presented Lie algebra: its generators are given by the dual of the first de Rham cohomology vector space, and the defining quadratic relations are determined by the cup product on cohomology.
In the non-Kähler case, this no longer holds; however, strong homotopy-theoretic restrictions still apply for surfaces. We prove the following:

\theoremstyle{theorem}
\newtheorem*{maint}{\normalfont\bfseries Main Theorem}\label{mainintro}
\begin{maint}
Let \(S\) be a compact complex non-Kähler surface. Then:
\begin{enumerate}
    \item There are no Massey products from \(H^1(S)\) to \(H^2(S)\) of length greater than three.
    \item The real Malcev completion of \(\pi_1(S)\) is determined by the bigraded nilpotent Lie algebra
    \[\bigl(\pi_1(S)/\Gamma_4 \bigr) \otimes \mathbb{R},\text{ where } \Gamma_2 = [\pi_1(S), \pi_1(S)]\text{ and }\Gamma_{i+1} = [\pi_1(S), \Gamma_i]\] is the lower central series.
    \item
    The real Malcev Lie algebra
    admits a presentation by generators $\{X_i,Y_i\}_{i=1}^k$ and $Z$, subject to the relations
 \[[Z,X_i]=0\quad,\quad[Z,Y_i]=0\quad,\quad \sum_{j=1}^k [[X_j,Y_j],X_i]=0\quad\text{ and }\quad\sum_{j=1}^k [[X_j,Y_j],Y_i]=0,\]
 for all $1\leq i\leq k$, where $k={1\over 2}(b^1-1)$.
\end{enumerate}
\end{maint}
In particular, for non-Kähler surfaces, the real Malcev completion of the fundamental group is determined by the first Betti number. This contrasts with the case of Kähler surfaces, where one must also take into account the structure of the cup product.

The theorem is
sharp in the following two directions. First, there exist compact complex surfaces with non-trivial triple Massey products; the Kodaira–Thurston manifold endowed with any left-invariant complex structure is such an example.
Second, the theorem does not extend to the neighbouring categories:
for instance, the filiform nilmanifold is a 4-dimensional almost complex manifold admitting symplectic but no integrable complex structures, and it exhibits a non-trivial quadruple Massey product. More generally, every finitely presented group arises as the fundamental grup of
a closed almost complex 4-manifold \cite{Kot}, of a symplectic 4-manifold \cite{Gompf}, and
of a complex projective surface with simple normal crossing singularities \cite{KaKo}. In particular, the theorem fails as soon as we relax the integrability condition or allow for singularities in the projective case.
Beyond the four-dimensional setting, it is also known that every finitely presented group can be realized in the category of compact complex threefolds, and more generally, in that of higher-dimensional complex manifolds (see, for example, \cite{ABCKT} for a proof). This once again evidences that complex surfaces are really special.

 The study of fundamental groups of complex manifolds has traditionally centered on the Kähler case, through the theory of \emph{Kähler groups}. These form a highly restrictive class among finitely presented groups, subject to strong constraints arising from Hodge theory and related analytic tools. In complex dimension two, there is a substantial body of work on fundamental groups of complex surfaces—both Kähler and non‑Kähler, with elliptic surfaces providing a particularly rich source of results (see \cite{ABCKT}, \cite{FrMo}). In particular,
 due to Kodaira's classification of surfaces, their fundamental groups are known to be quite restrictive.
 However, while rational homotopy techniques have been extensively exploited in the Kähler setting, comparable techniques appear to be largely absent in the study of non‑Kähler surfaces. We hope that our results will stimulate further progress on the study of fundamental groups, beyond the real Malcev completion, of compact non-Kähler surfaces.

 A key framework motivating this work is that of smooth algebraic varieties, whose rational homotopy models are enriched in \emph{mixed Hodge structures}, as shown by Morgan \cite{Morgan}. The presence of mixed Hodge structures at the cochain level and the restricted weights in cohomology imply that, for a smooth complex algebraic variety $X$, there are no $k$-tuple Massey products from $H^1(X)$ to $H^2(X)$ for $k > 4$. Dually, the real Malcev completion of $\pi_1(X)$ is determined by the bigraded nilpotent Lie algebra
$
\bigl(\pi_1(X)/\Gamma_5\bigr) \otimes \mathbb{R}$.
This provides strong homotopical restrictions on which finitely presented groups can occur as fundamental groups of smooth complex varieties.

Our strategy for proving the Main Theorem is to construct a 1-model of the de Rham algebra in the sense of rational homotopy theory, together with a weight decomposition. This commutative differential graded algebra encodes the real homotopy type of the manifold in low degrees. When $b^1=1$ this model is just the free algebra in one generator and so not so exciting. In this case we have 1-formality, Massey products vanish in all lengths $\geq 3$ and the real Malcev Lie algebra associated to the fundamental group is just the abelian one-dimensional Lie algebra.

When $b^1>1$, our model is described as follows.
Denote by $H_{BC}^{\leq 1}$ the sub-algebra of real Bott-Chern cohomology generated in degree $\leq 1$.
A 1-model for $\Aa_{\dR}(S)$ is then given by:
\[\Mm_\RR=H_{BC}^{\leq 1}\otimes\Lambda(\gamma,\gamma^c),\text{ with }d\gamma=0\text{ and }d\gamma^c=[\beta]_{BC}.\]
Here $[\beta]_{BC}$ denotes a non-trivial real Bott-Chern class generating the
the 1-dimensional kernel of the map $H^{1,1}_{BC}\to H^{1,1}_\delb$ induced by the identity.
Setting the Bott-Chern generators to be of weight $1$ and $\gamma$, $\gamma^c$ to be of weight 2, we obtain
a 1-model with a weight decomposition or, more generally, with a mixed Hodge structure with very restricted weights, from which the assertions of the Main Theorem may be deduced. Note our results are stated over $\RR$, and it remains an open question whether some of our structure results also hold $\QQ$.

Some recent results related to the present approach are the models for Vaisman and Sasakian manifolds \cite{Tievsky}, \cite{CaBe} and those for transverse Kähler structures \cite{Transverse}.
Let us note that,
according to Kodaira's classification, non-Kähler surfaces fall into two broad classes: class \( \mathrm{VII} \) surfaces, all with \( b^1 = 1 \), and elliptic surfaces, for which \( b^1 \geq 1 \).
Belgun \cite{Belgun} showed that minimal non-Kähler elliptic surfaces admit Vaisman metrics. Consequently, by the above cited works, they possess a concise model based on basic cohomology. Our homotopy-theoretic conclusions could potentially be derived from this approach, as the 1-minimal model of a minimal surface is insensitive to blow-ups at points. The approach we describe in this note is more direct and concrete, as it specifically addresses the case of complex surfaces.
%and does not depend on Kodaira's classification.

\subsection*{Acknowledgments} I am grateful to Liviu Ornea and Misha Verbitsky for the discussions related to Lemma \ref{omegainj}.
I owe Geoffroy Horel the strategy to deduce the presentation of $(3)$ in the main Theorem, from the 1-minimal model.
Thanks also to Jaume Amorós, Jonas Stelzig, Scott Wilson and the Critical Math for the useful discussions and comments.

\section{Cohomology of compact non-Kähler surfaces}\label{seccohom}

Denote by $\Aa:=\Aa_{\dR}^*\otimes\CC=\bigoplus \Aa^{*,*}$ the complex de Rham algebra of a compact complex surface $S$. Its differential splits into components $d=\del+\delb$, where $\delb$ is the Dolbeault operator, of bidegree $(0,1)$ and $\del$ is its complex conjugate. Let
$H^*_d$ denote the complex de Rham cohomology of $S$ and
$H^{*,*}_\delb$ its Dolbeault cohomology.
In the compact case, both de Rham and Dolbeault cohomologies are finite-dimensional and we let $b^k:=\dim H_d^k$ and $h^{p,q}:=\dim H_\delb^{p,q}$.
The Frölicher spectral sequence of any compact complex surface degenerates at the first page, and so
\[H^k_d\cong\bigoplus_{p+q=k}H^{p,q}_\delb  \quad\text{ and }\quad b^k=\sum_{p+q=k} h^{p,q}.\]
Denote by $\Omega^{p}:=\Aa^{p,0}\cap\Ker\delb$ the space of holomorphic $p$-forms.
On a compact complex surface, every holomorphic $p$-form is closed (see for instance Lemma IV.2.1 of \cite{Barth}) and so for all $p\geq 0$ we have
\[H^{p,0}_\delb=\Omega^p=\Ker d\cap \Aa^{p,0}.\]
This gives a well-defined map 
\[H^{1,0}_\delb\oplus \overline{H^{1,0}_\delb}\lra H^1\]
which is in fact injective. In the Kähler case, this map is an isomorphism. In the non-Kähler case,
for which $h^{0,1}=h^{1,0}+1$,
there is a closed form $\gamma$ such that
\[H^1_d\cong H^{1,0}_\delb\oplus \ov{H^{1,0}_\delb}\oplus\langle[\gamma]\rangle.\]
We call the class $[\gamma]$ the \textit{purity defect}, in the sense that it prevents the cohomology of a compact non-Kähler surface to carry a pure Hodge structure.

We will also consider
Bott-Chern and Aeppli cohomologies, defined respectively as
\[H^{p,q}_{BC}:={{\Ker\del\cap\Ker\delb\cap\Aa^{p,q}}\over{\del\delb(\Aa^{p-1,q-1})}}
\quad\text{ and }\quad H^{p,q}_{A}:={{\Ker\del\delb\cap\Aa^{p,q}}\over{\del(\Aa^{p-1,q})+\delb(\Aa^{p,q-1})}}.
\]
One easily verifies that 
\[H^{1,0}_{BC}\cong H^{1,0}_\delb=\Omega^1\text{ and }H^{0,1}_{BC}\cong \overline{{H}^{1,0}_\delb}=\ov\Omega^1.\]
Note the identity induces a map $H^{1,1}_{BC}\to H^{1,1}_\delb$. Also, there is a map
$\del:H^{0,1}_\delb\lra H^{1,1}_{BC}$
induced by the component $\del$ of the exterior differential.
The following result is essentially due to Lamari \cite{La}. We refer to \cite{Verbook} or \cite{Verb}
for a proof of this precise statement.

\begin{lemm}\label{Verblemma}
For any  compact non-Kähler surface there is an exact real form $\beta\in \Aa^{1,1}$ such that $[\beta]_{BC}\neq 0$ is non-trivial in Bott-Chern cohomology and:
\begin{enumerate}
 \item The class $[\beta]_{BC}$ is in the image of the map $\del:H^{1,0}_\delb\lra H^{1,1}_{BC}$.
 \item The map $H^{1,1}_{BC}\to H^{1,1}_\delb$ is surjective and its kernel is generated by $[\beta]_{BC}$.
 This gives a short exact sequence 
 \[0\to \langle [\beta]_{BC}\rangle\to H^{1,1}_{BC}\to H^{1,1}_\delb\to 0.\]
\end{enumerate}
\end{lemm}

We may describe the purity defect as follows.
By (1) of Lemma \ref{Verblemma}, there is a $\delb$-closed form $\ov\eta\in \Aa^{0,1}$ such that $[\beta]_{BC}=[\del\ov\eta]_{BC}$. Therefore we may write
\[\beta=\del\ov\eta+\del\delb f\]
for some smooth function $f$. Letting $\ov\theta:=\ov\eta+\delb f$ we  obviously have $d\ov\theta=\del\ov\theta=\beta$ and, since $\beta$ is real, we have that the complex conjugate $\theta$ of $\ov\theta$ satisfies $d\theta=\delb\theta=\beta$. The form
\[\gamma:={1\over 2\mathbf{i}}(\theta-\ov\theta)\]
is real, closed and satisfies $d^c\gamma=\beta$, where $d^c:=\mathbf{i}(\delb-\del)$. It thus generates the purity defect.

 With the above choices for $\theta$, $\ov\theta$ and $\beta$, we obtain a decomposition of the double complex of a complex non-Kähler surface into indecomposable double complexes as depicted below.
 
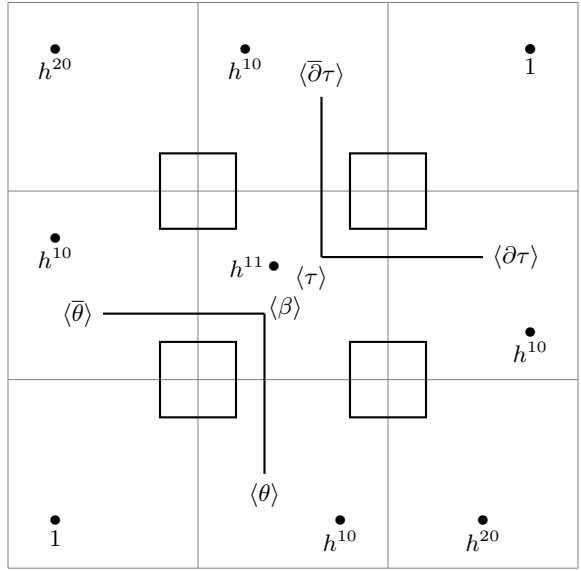
\begin{center}
\begin{tikzpicture}
% grid lines
\draw[very thin, gray]  (0,0) -- (0,7.5);
\draw[very thin, gray]  (0,0) -- (7.5,0);
\draw[very thin, gray]  (0,7.5) -- (7.5,7.5);
\draw[very thin, gray]  (7.5,0) -- (7.5,7.5);
\draw[very thin, gray]  (0,2.5) -- (7.5,2.5);
\draw[very thin, gray]  (2.5,0) -- (2.5,7.5);
\draw[very thin, gray]  (0,5) -- (7.5,5);
\draw[very thin, gray]  (5,0) -- (5,7.5);

% rectangles
\draw[thick] (2,2) rectangle (3,3);
\draw[thick] (4.5,4.5) rectangle (5.5,5.5);
\draw[thick] (4.5,2) rectangle (5.5,3);
\draw[thick] (2,4.5) rectangle (3,5.5);

% arrows
\draw[thick] (1.25,3.375) -- (3.375,3.375);
\draw[thick] (3.375,1.25) -- (3.375,3.375);

\draw[thick] (4.125,4.125) -- (6.25,4.125);
\draw[thick] (4.125,4.125) -- (4.125,6.25);

% labels
\draw (1.25,3.375) node[left] {\small{$\langle\ov\theta\rangle$}};
\draw (3.375,1.25) node[below] {\small{$\langle\theta\rangle$}};
\draw (3.3,3.45) node[right] {\small{$\langle\beta\rangle$}};
\draw (3.65,3.85) node[right] {\small{$\langle\tau\rangle$}};
\draw (4.125,6.25) node[above] {\small{$\langle\delb \tau\rangle$}};
\draw (6.25,4.125) node[right] {\small{$\langle\del \tau\rangle$}};

% bullets and labels

\draw (6.875,6.875) node {\small{$\bullet$}};
\draw (6.875,6.875) node [below] {\small{$1$}};

\draw (.625,.625) node {\small{$\bullet$}};
\draw (.625,.625) node [below] {\small{$1$}};

\draw (.625,6.875) node {\small{$\bullet$}};
\draw (.625,6.875) node [below] {\small{$h^{20}$}};

\draw (6.25,.625) node {\small{$\bullet$}};
\draw (6.25,.625) node [below]{\small{$h^{20}$}};

\draw (6.875,6.875) node {\small{$\bullet$}};
%\draw (6.875,6.875) node [below]{\small{$\beta\theta\ov\theta$}};

\draw (4.375,.625) node {\small{$\bullet$}};
\draw (4.375,.625) node[below] {\small{$h^{10}$}};

\draw (.625,4.375) node {\small{$\bullet$}};
\draw (.625,4.375) node [below]{\small{$h^{10}$}};

\draw (3.5,4) node {\small{$\bullet$}};
\draw (3.5,4) node[left] {\small{$h^{11}$}};

\draw (3.125,6.875) node {\small{$\bullet$}};
\draw (3.125,6.875) node[below] {\small{$h^{10}$}};

\draw (6.875,3.125) node {\small{$\bullet$}};
\draw (6.875,3.125) node [below] {\small{$h^{10}$}};
\end{tikzpicture}
\end{center}

In the diagram, \textit{dots} represent vector spaces generated by forms whose differential vanishes and which are not in the image of $\del$ nor of $\delb$. The number of dots in each bidegree is indicated. All numbers are determined by the topology of $S$: denote by $\sigma=b^+-b^-$ the signature of the surface.
Then
\[h^{10}={1\over 2}(b^1-1),\,h^{11}=b^-\text{ and }h^{20}=b^+/2.\]
 A \textit{square} represents a vector space generated by a form $\alpha$ in its lower-left corner with $\del\delb \alpha\neq 0$. There are infinitely-many such squares in each bidegree. The two open corners represent the spaces generated by  $[\beta]_{BC}$ and a non-trivial Aeppli class $[\tau]_A$ respectively.
 There is only one of each and, in the next section, we will show that, if $\Omega^1\neq 0$, then one can choose $\beta=\mathbf{i}(x\wedge \ov x)$, where $x\in\Omega^1$ and $\tau=\theta\wedge\ov\theta$.

The above diagram in bidegree $(0,0)$ summarizes the well-known fact that, for a compact connected complex manifold,
any smooth complex-valued function $f$ satisfying $\del\delb f=0$ must be constant (see for instance Lemma IV.2.2 of \cite{Barth}).
Likewise, in bidegree $(1,0)$ the diagram tells us that if $z\in\Aa^{1,0}$ satisfies $\del\delb z=0$, then $z=x+a\theta+\del f$ where $x\in\Omega^1$, $a\in\CC$ and $f\in\Cc^\infty$.
Other bidegrees in the diagram should be interpreted similarly. We refer to \cite{Stelzig} for a proof of the fact that any bicomplex admits such a decomposition, and to \cite{SW} for some specific examples, including that of complex surfaces.

\begin{rema}As shown in \cite{DGMS}, the double complex decomposition associated to any compact Kähler manifolds consists only in dots and squares, and this is equivalent to the $\del\delb$-condition.
 In our case, we have two extra $\mathrm{L}$-shaped summands, and so the above diagram can be thought as a perturbation of the $\del\delb$-condition that is satisfied for non-Kähler surfaces. In \cite{SW}, this is called the $\del\delb+3$-condition, since shapes with 3 vertices are allowed. Aside from complex surfaces, there are other families of complex manifolds satisfying this condition, such as Vaisman manifolds.
\end{rema}

\section{Multiplicative structure}\label{secmulti}

We collect some multiplicative properties of the differential forms on a compact non-Kähler surface.
First, we have (c.f. Section 25.2.4 of \cite{Verbook} and Lemma 6.3 of \cite{La}):

\begin{lemm}\label{Verbmulti}
Let $x,y\in\Omega^1$ be holomorphic $1$-forms. Then:
\begin{enumerate}
 \item The products $x\wedge y$ and $\ov{x}\wedge \ov{y}$ vanish and the product $x\wedge\ov y$ is $d$-exact.
 \item There is a smooth function $f$ and a constant $c\in\CC$ such that
  \[x\wedge \ov y=\del\delb f + c\beta,\text{ with }c\in\CC,\]
  where $\beta$ is a real $d$-exact representative of the non-trivial Bott-Chern class of Lemma \ref{Verblemma}.
\item If $x=y\neq 0$ then, in the above formula, we have $c\neq 0$.
\end{enumerate}
\end{lemm}
\begin{proof}
The first statement is proven in Proposition 2.12 of \cite{Verbook}. We give a proof for completeness.
Note first that, unless $x\wedge y=0$, we have
\[\int_S x\wedge y\wedge \ov x\wedge \ov y>0.\]
Also, the intersection form on $H^{1,1}_d$ is negative-definite in the non-Kähler case (see \cite{Barth}, Theorem IV.2.14), and so any class $[\mu]\in H^{1,1}_d$ with
\[\int_S \mu\wedge \ov \mu=0\]
is necessarily trivial.
Now, the form $\mu:=x\wedge \ov x$ is closed and so it defines a cohomology class in $H^{1,1}_d$ which obviously satisfies $\mu\wedge\ov \mu=0$. Therefore $[\mu]$ is the trivial cohomology class and so $\mu$ is exact. In particular,
 $x\wedge \ov x$ and $y\wedge \ov y$ are exact and so the first integral above vanishes by Stokes. This proves $x\wedge y=0$.
The proof that
 $x\wedge \ov y$ is exact follows analogously. This proves (1).

Statement (2) follows directly from (1) and the double complex decomposition.
 Let us prove (3). By (2) we may write
 \[x\wedge \ov x=\del\delb f+c\beta\]
 for some funcion $f$ and constant $c$.
 Assume that $c=0$. Let $w_g$ be the $(1,1)$-form associated to a Gauduchon metric (so it satisfies $\del\delb w_g=0$). By Stokes' Theorem, we have
 \[\int_S x\wedge \ov x\wedge w_g=\int_S \del\delb f\wedge w_g=\int_Sf\wedge \del\delb w_g=0.\]
 Note however that $x\wedge \ov x$ is semi-positive and, since $w_g$ is a positive definite Hermitian form, we have 
 \[\int_S x\wedge\ov x\wedge w_g>0.\]
 This is a contradiction, and so we must have $c\neq 0$.
\end{proof}

Recall from the previous section that
$H^1_d\cong \Omega^1\oplus\ov\Omega^1\oplus [\gamma]$,
where $[\gamma]$ denotes the purity defect, and we may choose a representative
\[\gamma={1\over 2i}(\theta-\ov\theta), \text{ with }d\theta=d\ov\theta=\beta\,;\, \theta\in\Aa^{1,0}.\]
where  $\beta\in \Aa^{1,1}$ is the exact real form of Lemma \ref{Verblemma}.
The following property is proven in \cite{Verbook}, \cite{Verb} for complex surfaces with no complex curves. From Kodaira's classification, it is true a posteriori in general. A direct proof would be desirable.

\begin{lemm}\label{omegainj}The products $x\wedge \theta$ and $x\wedge \ov\theta$ are non-zero for any non-zero holomorphic $1$-form $x$.
 \end{lemm}

In the case $\Omega^1\neq 0$, the above lemmas allow to choose representatives
$\beta=\mathbf{i}(x\wedge\ov x)$, where $0\neq x\in\Omega^1$ and $\tau=\theta\wedge \ov \theta$ for the non-trivial Bott-Chern and Aeppli classes in bidegree $(1,1)$ depicted in the double complex decomposition of the previous section. These are dual to each other, in the sense that they fit in the non-degenerate pairing
\[H^{1,1}_{BC}\times H^{1,1}_A\lra \CC\,;\,([\beta]_{BC},[\tau]_A)\mapsto \int_S \beta\wedge\tau =\mathbf{i}\int_S x\wedge \ov x\wedge \theta\wedge\ov \theta\neq 0.\]
In particular, from these identifications and the double complex decomposition, one may easily describe the
product structures of Dolbeault, de Rham and Bott-Chern cohomologies.
The following result will be particularly useful.
Denote by $H_{BC}^{\leq 1}$ the sub-algebra of $H_{BC}^{*,*}$ generated in degrees $\leq 1$. We have:
\begin{lemm}\label{Bottchernbasis}
Assume that $k:=\mathrm{dim\,}\Omega^1>0$.
The algebra $H_{BC}^{\leq 1}$ is generated by elements
$\{x_i,\ov x_j\}_{j=1}^k$ with the only non-trivial products given by
$x_i\cdot\bar{x}_i=x_1\cdot \bar{x}_1$ for all $i\neq 1$.
\end{lemm}
\begin{proof}
Recall first we have isomorphisms $H^{1,0}_{BC}\cong \Omega^1$ and $H^{0,1}_{BC}\cong \ov\Omega^1$.
Take a basis $\{x_j\}$ of $\Omega^1$. By (2) of Lemma \ref{Verbmulti} we have
\[x_i\wedge x_j=0,\, \ov x_i\wedge \ov x_j=0\text{ and }x_i\wedge \ov x_j=c_{ij}\beta+\del\delb f_{ij},\]
where $\beta$ is a representative of $[\beta]_{BC}\in H^{1,1}_{BC}$ and $f_{ij}\in\Cc^\infty$.
By (3) of Lemma \ref{Verbmulti}, we may assume $\beta=\mathbf{i}(x_1\wedge \ov x_1)$,
so in particular we have $[\beta]_{BC}\in H^{\leq 1}_{BC}$. This allows to describe $H^{\leq 1}_{BC}$ as the algebra generated by
$\{[x_i],[\ov x_j]\}$ with the product structure
\[[x_i]\cdot[x_j]=0,\, [\ov x_i]\cdot [\ov x_j]=0\text{ and }[x_i]\cdot [\ov x_j]=c_{ij}[\beta]_{BC},\]
with $c_{ii}\neq 0$ and $c_{ji}=-\ov{c_{ij}}$.
Moreover, since $c_{11}=-\mathbf{i}$ and $\mathbf{i}(x\wedge \ov x)$ is always semi-positive for any holomorphic form $x$, we have
$c_{ii}\in -\mathbf{i}\RR^+$.
Therefore the matrix $\mathbf{i}c_{ij}$ is symmetric and positive-definite, so changing basis we may assume $c_{ij}=0$ for all $i\neq j$ and $c_{ii}=-\mathbf{i}$ for all $i$.
\end{proof}

We will also use the following lemma, which is a direct
 consequence of Lemma \ref{omegainj} together with the double complex decomposition:

\begin{lemm}\label{variousproducts}
Let $f\in\Cc^\infty$ such that
$\del\delb f\wedge\Omega^1=\del\delb f\wedge\ov\Omega^1=0$ and let
 $x\in\Omega^1$. Then:
\begin{enumerate}
 \item $\delb f\wedge\ov x=0$.
 \item $\delb f\wedge x=\del\delb g$ for some $g\in\Cc^\infty$ such that $\del\delb g\wedge\Omega^1=\del\delb g\wedge \ov\Omega^1=0$.
\end{enumerate}
\end{lemm}
\begin{proof}
Since $\delb f\wedge \ov x=\delb(f\wedge \ov x)\in\mathrm{Im}(\delb)$ we may write $\delb f\wedge x=\delb w$ for some $w\in\Aa^{0,1}$. But since $\del\delb f\wedge x=0$ this gives $\del\delb w=0$ which implies $w=0$.
For the second identity, note we may write 
\[\delb f\wedge x=\del\delb g+a\beta+\delb z,\]
for some constant $a$, and where $z\in \Aa^{1,0}$ satisfies $\del\delb z\neq 0$ if $z\neq 0$. Since
\[\del(\delb f\wedge x)=\del\delb f\wedge x=0,\]
it follows that $z=0$.
Assume that $a\neq 0$. Then we  may choose \[\theta={1\over a} (f\cdot x+\del g),\]
and so $\theta\cdot x= \del (g\cdot x)$.
Since $\delb (\theta\cdot x)=0$ this implies $g\cdot x=0$ and so $\theta\wedge x=0$, which contradicts Lemma \ref{omegainj}. Therefore $a=0$ and the lemma follows.
\end{proof}

% {In case its useful, another multiplicative property is
% $f_{ij}\cdot x_k\in \Omega^1\oplus \del\Cc^\infty$.}

\section{A 1-model for non-Kähler surfaces}

Let us first recall the notion of 1-model from rational homotopy theory:
\begin{defi}Let $f:A\to B$ be a map of commutative dg-algebras. Then $f$ is said to be a \textit{1-quasi-isomorphism} if $H^1(f)$ is an isomorphism and $H^2(f)$ is injective.
A \textit{1-model} for a commutative dg-algebra $A$ is a commutative dg-algebra $\Mm$ together with a string of 1-quasi-isomorphisms connecting $\Mm$ and $A$.
\end{defi}

When $b^1=1$, a 1-model for $\Aa_{\dR}(S)$ is given by $\Lambda(\gamma)$, where $[\gamma]$ generates $H^1_d$. We assume from now on that $b^1>1$.
This implies that the space of holomorpic 1-forms $\Omega^1$ is non-trivial and so Bott-Chern cohomology is also non-trivial in degree 1.

We will first build a complex model, and show later how the same construction works over the real numbers.
Consider the commutative differential bigraded algebra given by
\[\Mm:=H_{BC}^{\leq 1}\otimes\Lambda(\Theta,\ov \Theta),\]
where
$\Theta$ and $\ov \Theta$ denote elements of bidegree $(1,0)$ and $(0,1)$ respectively and
$H_{BC}^{\leq 1}$ is the algebra described in Lemma \ref{Bottchernbasis}.
Define a differential on $\Mm$ by letting
\[d\Theta=d\ov \Theta=[\beta]_{BC}.\]

\begin{theo}\label{theomodel}Let $S$ be a compact non-Kähler surface with $b^1>1$.
Then the differential bigraded algebra $(\Mm,d)$
is a 1-model for $\Aa_{\dR}(S)\otimes\CC$.
\end{theo}

Before going through the actual proof, let us explain the strategy by analogy with the proof of formality for compact Kähler manifolds. In that case, de Rham cohomology is shown to be a model of the complex de Rham algebra by considering the inclusion and projection maps
\[(H_\delb,0)\stackrel{p}{\longleftarrow}(\Ker \delb,\del)\stackrel{i}{\longrightarrow}(\Aa_{\dR}\otimes\CC,d)\]
Then the $\del\delb$-lemma ensures that these maps are quasi-isomorphisms.
In our case, the $\del\delb$-lemma needs to be replaced by the various rules given by the double complex decomposition described in Section \ref{seccohom}, and $H_{\delb}$ is replaced by the algebra $\Mm$ defined above.
The intermediate algebra $(\Ker \delb,\del)$ will be in spirit very similar, with some minor adjustments that are described in the proof below.

\begin{proof}
We construct a roof of 1-quasi-isomorphisms connecting $\Aa:=\Aa_{\dR}(S)\otimes\CC$ and $\Mm$.
For that, let us first denote by $\Bb$ the sub-algebra of $\Aa$
generated by the following subspaces:
\[\Omega^0\oplus \Omega^1\oplus \ov\Omega^1\oplus \langle\ov\theta\rangle\oplus \delb \Dd\oplus \del\delb \Dd,\]
where
\[\Dd:=\{f\in\Cc^\infty; \del\delb f\wedge\Omega^1=\del\delb f\wedge \ov\Omega^1=0\}.\]
We have $d(\Bb)\subseteq \Bb$ by construction, and $\beta=\mathbf{i}(x_1\wedge \ov x_1)\in\Bb$.
Now, consider the dg-algebra given by
\[\Bb\otimes\Lambda(\Theta)\text{ with }d\Theta=\beta.\]
The generators of this algebra are indicated below with bullets.

 \begin{center}
 \begin{tikzpicture}
\draw[very thin, gray]  (0,0)   -- (0,6);
\draw[very thin, gray]  (0,0)   -- (6,0);
\draw[very thin, gray]  (0,6)   -- (6,6);
\draw[very thin, gray]  (6,0)   -- (6,6);
\draw[very thin, gray]  (0,2)   -- (6,2);
\draw[very thin, gray]  (2,0)   -- (2,6);
\draw[very thin, gray]  (0,4)   -- (6,4);
\draw[very thin, gray]  (4,0)   -- (4,6);

\draw[thick] (1.5,1.5) rectangle (2.5,2.5);
\draw[thick] (3.5,3.5) rectangle (4.5,4.5);
\draw[thick] (3.5,1.5) rectangle (4.5,2.5);
\draw[thick] (1.5,3.5) rectangle (2.5,4.5);

\draw[thick]  (1,2.9)   -- (2.9,2.9);
\draw[thick]  (2.9,1)   -- (2.9,2.9);

\draw[thick]  (3.1,3.1)   -- (5,3.1);
\draw[thick]  (3.1,3.1)   -- (3.1,5);

\draw (1,2.9) node {\small{$\bullet$}};
\draw (1,2.9) node[left] {\small{$\langle\ov\theta\rangle$}};

\draw (2.9,1) node {\small{$\bullet$}};
\draw (2.9,1) node[below] {\small{$\langle\Theta\rangle$}};

\draw (3.5,1) node {\small{$\bullet$}};
\draw (3.5,1) node[below] {\small{$\,\Omega^1$}};

\draw (.8,3.5) node {\small{$\bullet$}};
\draw (.8,3.5) node[left] {\small{$\ov\Omega^1$}};

\draw (1,1) node {\small{$\bullet$}};
\draw (1,1) node[below] {\small{$\Omega^0$}};

\draw (1.5,2.5) node {\small{$\bullet$}};
\draw (1.5,2.5) node[left] {\small{$\delb \Dd$}};

\draw (2.5,2.5) node {\small{$\bullet$}};

\end{tikzpicture}
\end{center}

Define a map $\varphi:\Bb\otimes\Lambda(\Theta)\to \Aa$ by the inclusion on the subspaces of $\Aa$ and let
$\varphi(\Theta):=\theta$.
This makes $\varphi$ compatible with the differentials.
Moreover, $\varphi$ is clearly multiplicative.

To define a map $\psi:\Bb\otimes\Lambda(\Theta)\to \Mm$,
we send $\Omega^0$, $\Omega^1$ and $\ov\Omega^1$ to their Bott-Chern classes and
$\Theta$ and $\ov\theta$ are sent to their obvious namesakes. The spaces $\delb \Dd$ and $\del\delb \Dd$ are sent to 0.

The map $\psi$ is compatible with differentials, as
\[\psi(d\Theta)=\psi(\beta)=[\beta]_{BC}=d\psi(\Theta).\]
To prove that $\psi$ is multiplicative it suffices to verify the following:
\begin{itemize}
 \item Given $x,y\in\Omega^1$, then
 by (2) of Lemma \ref{Verblemma}, we have
$x\wedge \ov y=\del \delb g+c\beta$
and so
\[\psi(x)\wedge\psi(\ov y)=[x\wedge \ov y]_{BC}=c[\beta]_{BC}=\psi(x\wedge \ov y).\]
\item Let $x\in \Omega^1$ and $f\in\Dd$. Then by (2) of Lemma \ref{variousproducts} we have $\delb f\wedge x=\delb g$, with $g\in\Dd$.
Therefore we have
\[\psi(\delb f)\wedge\psi(x)=\psi(\delb f\wedge x)=0.\]
Moreover, by (1) of Lemma \ref{variousproducts} we have $\delb f\wedge \ov x=0$ and so
\[\psi(\delb f)\wedge\psi(\ov x)=\psi(\delb f\wedge \ov x)=0.\]
\end{itemize}

This proves that $\psi$ is multiplicative.

We next prove that the maps $\varphi$ and $\psi$ are 1-quasi-isomorphisms.

\underline{$H^1_d(\varphi)$ is an isomorphism:}
An element of degree 1 in $\Bb\otimes\Lambda(\Theta,\ov\Theta)$
has the form
\[\xi=x+\ov y+\delb g+a\Theta+b\ov\theta,\]
were $x,y\in\Omega^1$, $a,b\in\CC$ and $g\in\Dd$.
Therefore its differential is
\[d\xi=\del\delb g+(a+b)\beta.\]
Since $\beta$ defines a non-trivial Bott-Chern class, for $\xi$ to be closed we must have $g=0$  and $b=-a$. Therefore we may write
\[\xi=x+\ov y+a(\Theta-\ov\theta).\]
Now, assume $\varphi(\xi)$ is $d$-exact, for instance $\xi=dh$.
Then
\[\delb \xi=a\beta=\del\delb h.\]
Since $[\beta]_{BC}\neq 0$, we must have $a=0$ and $h=0$, which gives $x=\ov y=0$.
This proves injectivity. Surjectivity is straightforward, noting that $\varphi(\Theta-\ov\theta)$ generates the purity defect.

\underline{$H^1_d(\psi)$ is an isomorphism:}
Let $\xi=x+\ov y+a(\Theta-\ov\theta)$ be a $d$-closed element of degree 1 as before. Then $\psi(\xi)=[x]_{BC}+[y]_{BC}+a(\Theta-\ov\theta)$ which is never $d$-exact. Surjectivity is straightforward, since
$H^1(\Mm)$ is isomorphic to $\Omega^1\oplus \ov\Omega^1\oplus \langle\Theta-\ov\Theta\rangle$.

\underline{$H^2_d(\varphi)$ is injective:} Let $\xi$ be a $d$-closed element of $\Bb\otimes\Lambda(\Theta,\ov\Theta)$ of degree 2.
Writing $\xi=\xi_0+\xi_1+\xi_2$, with $\xi_i$ of bidegree $(i,2-i)$, the condition $d\xi=0$ gives
\[\del \xi_0+\delb \xi_1=0\text{ and }\del \xi_1+\delb \xi_2=0.\]
We will show that, in fact, we have $\del \xi=\delb\xi=0$.
For this, note first that the only generator with $\delb\neq 0$ is $\Theta$. Assume that $\xi$ has a summand that is multiple of $\Theta$. Such a summand will be of the form
\[(x+\ov y+\delb f+a\ov\theta)\wedge\Theta,\]
with $x,y\in\Omega^1$, $f\in\Dd$, and  and $a\in\CC$.
The differential of $\xi$ therefore has a summand that is multiple of $\Theta$:
\[(\del\delb f+a\beta)\wedge\Theta.\]
Therefore, for $\xi$ to be closed, we must have $\del\delb f+a\beta=0$. Since
 $\beta$ defines a non-trivial Bott-Chern class, this forces $a=0$ and $\del\delb f=0$ (which implies $f=0$). This shows that the only summands of $\xi$ that are multiples of $\Theta$ are in $(\Omega^1\oplus\ov\Omega^1)\wedge\Theta$, whose differential is trivial, since
 $\Omega^1\wedge\beta=0$ and $\ov\Omega^1\wedge \beta=0$ by assumption.
 This proves $\delb\xi=0$ and, since $\xi$ is $d$-closed, we also have $\del\xi=0$.
 Now, since $\del\xi_1=0$ it follows that $\varphi(\xi_1)$ is not $\delb$-exact. Similarly, $\varphi(\xi_2)$ ois not $\del$-exact.
 Now, $\xi_1$ is of the form
 \[\xi_1=x\wedge \ov y+\Theta\wedge \ov y+\delb f\wedge x+\del\delb g,\]
 with $f,g\in \Dd$. Using the multiplicative properties of Lemmas \ref{Verbmulti} and \ref{variousproducts} it follows that $\xi_1$ is $d$-cohomologous to
 $\Theta\wedge \ov y$. By Lemma \ref{omegainj}, $\varphi(\Theta\wedge \ov y)=\theta\wedge\ov y$ is never $d$-exact.

\underline{$H^2_d(\psi)$ is injective:}  Since the only non-trivial differentials in $\Mm$ are $d\Theta=d\ov\Theta=\beta$, it suffices to note that if $\psi(\xi)=c[\beta]_{BC}$ then we have $[\xi]_{BC}=c[\beta]_{BC}$ and so $\xi=\del\delb g+c\beta$, which is $d$-exact in $\Bb\otimes\Lambda(\Theta,\ov\Theta)$.
\end{proof}

The above theorem gives a complex model, since we are using the complex decomposition of forms and the complex operators $\del$ and $\delb$.
However, the same proof adapts to give a real model as follows:
on the real de Rham algebra $\Aa_{\dR}(S)$, consider the real operator $d_c=\mathbf{i}(\delb-\del)$.
Then $(\Aa_{\dR}(S),d,d_c)$ is a double complex and the double complex decomposition of Section \ref{seccohom} applies as well. Bott-Chern cohomology is defined as
\[H_{BC}^k:={{\Ker d\cap\Ker d_c\cap\Aa^{k}}\over{dd_c(\Aa^{k-1})}}.\]
By Lemma \ref{Bottchernbasis} and taking a real basis $u_i=x_i+\ov x_i$ and $v_i=\mathbf{i}(x_i-\ov{x}_i)$,
the algebra $H^{\leq 1}_{BC}$  is generated by $\{u_i,v_i\}_{i=1}^k$, where $k=\dim \Omega^1$, and the only non-trivial products are
$u_i\cdot v_i=[\beta]_{BC}$ for all $i$.

The non-trivial Bott-Chern class $[\beta]_{BC}$ is real, and
and
so we may define a real dg-algebra by letting
\[\Mm_\RR:=H_{BC}^{\leq 1}\otimes\Lambda(\gamma,\gamma^c)\text{ with }d(\gamma)=0\text{ and }d(\gamma^c)=[\beta]_{BC}.\]

The above proof exchanging the roles of $\del$ and $\delb$ by $d$ and $d_c$ immediately gives a real model:

\begin{theo}\label{theomodelreal}Let $S$ be a compact non-Kähler surface with $b^1>1$. Then the differential graded algebra $(\Mm_\RR,d)$
is a 1-model for $\Aa_{\dR}(S)$.
\end{theo}

\begin{rema}
Note that $H_{BC}^{\leq 1}$ is just a model for a compact curve of genus $k=\dim\,\Omega^1$, and so $\Mm_\RR$ agrees with a model of a torus fibration over such a curve. This is consistent with the fact that minimal non-Kähler surfaces with $b^1>1$ are elliptic \cite{Belgun}.
\end{rema}

\section{Weights in homotopy theory}\label{secweights}
In this section we review how the theory of weight decompositions gives non-trivial consequences on the homotopy type of a differential graded algebra. This will later be applied to the 1-model constructed in the previous section.

\begin{defi}
 A \textit{weight decomposition} of a commutative dg-algebra $(A,d)$ is a direct sum decomposition $A=\bigoplus_{p\in\ZZ} A_p$ such that 
 \[d:A_p^n\to A_p^{n+1}\text{ and }A^n_p\times A^m_q\to A^{n+m}_{p+q}.\]
\end{defi}
Note that a weight decomposition defines a bigrading on cohomology:
\[H^n_p(A):=H^n(A_p).\]
 A weight decomposition is said to be \textit{pure} if $H^n_p(A)=0$ for all $p\neq n$. 
The following is well-known (see for instance Proposition 3.3 of \cite{Thom}):
\begin{prop}\label{formalpure}
 If a (commutative) dg-algebra admits a pure weight decomposition, then it is formal. In particular, all higher Massey products vanish.
\end{prop}
We will consider weaker properties of weight decompositions that, while do not imply formality, still imply the vanishing of certain Massey products.
% One mild but non-trivial property is the following:
% \begin{defi}
%  A weight decomposition is said to be \textit{positive} if
%  \[A^n=\bigoplus_{p>0}A^n_p\text{ for }n>0\text{ and }A^0=A^0_0.\]
% \end{defi}
% Note that every commutative dg-algebra admits a trivial weight decomposition of weight 0, but this is clearly non-positive. In fact, positive weight decompositions give non-trivial homotopical obstructions on topological spaces (see \cite{Body}).
% One main family of spaces admiting positive weights is that of smooth complex algebraic varieties, as shown by Morgan in \cite{Morgan}.
% More precisely, using the theory of mixed Hodge structures, Morgan
% showed that smooth varieties admit models whose weights in cohomology are determined by the weight filtration of their mixed Hodge structures.
% This gives strong homotopical consequences,
% such as the vanishing of Massey products of length $\geq 4$ from $H^1\to H^2$ and restrictions on the rational completion of the fundamental group of smooth varieties.
% We will prove analogous results for compact complex surfaces.
Let us first recall the definition of Massey products.

\begin{defi}
Let $A$ be a dg-algebra and $x_1,\cdots,x_k\in H^*(A)$ cohomology classes, with $k\geq 3$. A \textit{defining system} for $\{x_1,x_2,\cdots,x_k\}$ is a collection of elements $\{x_{i,j}\}$, for $1\leq i\leq j\leq k$ with $(i,j)\neq (1,k)$
where $x_i=[x_{i,i}]$ and 
\[d(x_{i,j})=\sum_{q=i}^{j-1} (-1)^{|x_{i,q}|} x_{i,q} x_{q+1,j}.\]
Consider the cocycle 
\[\gamma(x_{i,j}):=\sum_{q=1}^{k-1}(-1)^{|x_{1,q}|} x_{1,q} x_{q+1,k}.\]
The \textit{$k$-tuple Massey product} $\langle x_1,\cdots,x_k\rangle$ is defined to be the set of all cohomology classes
$[\gamma(x_{i,j})]$, for all possible defining systems. 
A Massey product is said to be \textit{trivial} if the trivial cohomology class belongs to its defining set.
\end{defi}

\begin{rema}
Note that the triple Massey product $\langle x_1,x_2,x_3\rangle$ is empty unless
$x_1x_2=0$ and $x_2x_3=0$. For $k>3$ one similarly asks that some $q$-tuple Massey products, with $q<k$,
are trivial in a certain compatible way, so that at least one defining system exists.
\end{rema}

The following result is tailored to our situation with compact non-Kähler surfaces:

\begin{lemm}
 Let $A$ be a dg-algebra with a weight decomposition 
 $A=\bigoplus A_p$. Assume that the induced weights in cohomology satisfy 
 \[H^1(A)=H^1_1\oplus H^1_2\text{ and }H^2(A)=H^2_2\oplus H^2_3.\]
 Then 
 there are no $k$-tuple Massey products from $H^1(A)\to H^2(A)$, for $k>3$.
\end{lemm}
\begin{proof}Let $k>3$.
Assume that $x_i\in H^1_{w_i}(A)$, for $i=1,\cdots,k$.
We will show that $\langle x_1,\cdots,x_k\rangle$ is trivial. 
A defining system $\{x_{i,j}\}$ is given by elements of degree $1$ and weights $w(x_{i,j})=\sum_{q=i}^j w_i$. The above expression for $\gamma(x_{i,j})$ forces its weight to be
$w=\sum_{i=1}^k w_i$.
Therefore the $k$-tuple Massey product $\langle x_1,\cdots,x_k\rangle$ lives in
$H^{2}_w(A)$.
Now, since $w_i\in\{1,2\}$, we have $w\geq k>3$. But $H^2_p(A)=0$ for $p>3$.
\end{proof}

\begin{rema}
One may similarly prove other versions of the above lemma. In particular, in the case of smooth complex varieties, Morgan's theory of mixed Hodge structures on rational homotopy types gives the vanishing of Massey products of length $>4$ as, in that case, the weights in $H^k$ range from $k$ to $2k$.
\end{rema}

In our set-up, we are dealing with 1-models (rather than full models, given by quasi-isomorphic algebras). However, 
naturality is still valid for low-degree Massey products:
\begin{lemm}
 Let $f:A\to B$ be a 1-quasi-isomorphism and $x_1,\cdots,x_k$ cohomology classes in $H^1(A)$. Then
 \[f^*\langle x_1,\cdots,x_k\rangle=\langle f^* x_1,\cdots,f^* x_k\rangle.\]
\end{lemm}
\begin{proof}
This is proved in \cite{Kraines} for Massey products in any degree. For the case of 1-quasi-isomorphisms, the same proof remains valid if we restrict to degree 1 cohomology classes.
\end{proof}

We now turn to the real Malcev completion of the fundamental group of a topological space $X$.
Form the lower central series for $\pi_1(X)$:
\[\cdots\subset \Gamma_3\subset \Gamma_2\subset\pi_1(X)\]
where $\Gamma_2=[\pi_1(X),\pi_1(X)]$ and $\Gamma_{i+1}=[\pi_1(X),\Gamma_i]$. Taking quotients we get the Malcev completion of $\pi_1(X)$. It is given by the tower of nilpotent groups
\[\cdots\to \pi_1(X)/\Gamma_3\to\pi_1(X)/\Gamma_2\to \{1\}.\]
Each $\pi_1(X)/\Gamma_i$ is a nilpotent group of index $i$, and is a central extension of $\pi_1(X)/\Gamma_{i-1}$ by the abelian group $\Gamma_{i+1}/\Gamma_i$. For a field $\kk$, the $\kk$-Malcev completion of $\pi_1(X)$ is given by tensoring the above tower by $\kk$
(see \cite{Sullivan}).

Assume now that $X$ is a smooth manifold, and let $\mathcal{M}$ be a 1-minimal model of $\mathcal{A}_{\mathrm{dR}}(X)$.
By definition, the minimal 1-model may be written as a colimit of free commutative dg-algebras over $\RR$
\[\RR\to \Mm_1\to \Mm_2\to\cdots,\]
each $\Mm_i$ generated in degree $1$, and such that $d\Mm_i\subset \Mm_{i-1}$.
Dualizing the sets of generators we get a tower of $\RR$-Lie algebras, the \textit{real Malcev Lie algebra}
\[\cdots\to \Ll_2\to\Ll_1\to 0\]
and each $\Ll_{i+1}$ is a central extension of $\Ll_i$. A main result due to Sullivan \cite{Sullivan} states that this tower of Lie algebras is the tower of Lie algebras associated to the $\RR$-Malcev completion of $\pi_1(X)$.
The assumption that $X$ is smooth is only needed to make use of its de~Rham algebra. Alternatively, we could proceed using Sullivan's algebra of piecewise linear forms $\mathcal{A}_{\mathrm{PL}}(X)$, which is defined for any topological space. In this case, it is also possible to work over $\mathbb{Q}$ instead of $\mathbb{R}$.

\begin{prop}Assume that $\Aa_{\dR}(X)$ admits a 1-model $\Mm$ together with a weight decomposition.
\begin{enumerate}
 \item If the weight decomposition is pure, then the real Malcev completion of $\pi_1(X)$ is determined
 by the cup product $H^1(X)\times H^1(X)\to H^2(X)$ and its Malcev Lie algebra is quadratically presented.
 \item If the induced weights in cohomology satisfy
 $H^1(X)\cong H^1_1\oplus H^1_2$ and $H^2(X)\cong H^2_2\oplus H^2_3$
 then the real Malcev completion of $\pi_1(X)$ is determined by the bigraded nilpotent Lie algebra of $(\pi_1(X)/\Gamma_4)\otimes\RR$.
\end{enumerate}
\end{prop}
\begin{proof}
First, note that weight decompositions carry on to minimal models so we may assume that such 1-model is minimal (see for instance Lemma 3.2 of \cite{Thom}).
In the pure case, the minimal 1-model is formal by Proposition \ref{formalpure} and hence may be computed from the cup product $H^1(X)\times H^1(X)\to H^2(X)$. The second statement follows from the fact that
the associated graded Lie algebra has generators with weights -1 and -2 (corresponding to the negative values of the weights in $H^1$) and relations of weights -2 and -3 (corresponding to the negative values of the weights in $H^2$). Consequently, the graded Lie algebra modulo its fourth order commutators reconstructs the full tower. We refer to \cite{Morgan} for more details.
\end{proof}

\section{Main Theorem}

\begin{theo}
Let \(S\) be a compact complex non-Kähler surface. Then:
\begin{enumerate}
    \item There are no Massey products from \(H^1(S)\) to \(H^2(S)\) of length greater than three.
    \item The real Malcev completion of \(\pi_1(S)\) is determined by the bigraded nilpotent Lie algebra
    \[\bigl(\pi_1(S)/\Gamma_4 \bigr) \otimes \mathbb{R},\text{ where } \Gamma_2 = [\pi_1(S), \pi_1(S)]\text{ and }\Gamma_{i+1} = [\pi_1(S), \Gamma_i]\] is the lower central series.
    \item
    The real Malcev Lie algebra
    admits a presentation by generators $\{X_i,Y_i\}_{i=1}^k$ and $Z$, subject to the relations
 \[[Z,X_i]=0\quad,\quad[Z,Y_i]=0\quad,\quad \sum_{j=1}^k [[X_j,Y_j],X_i]=0\quad\text{ and }\quad\sum_{j=1}^k [[X_j,Y_j],Y_i]=0,\]
 for all $1\leq i\leq k$, where $k={1\over 2}(b^1-1)$.
\end{enumerate}
\end{theo}
\begin{proof}
 In view of Section \ref{secweights}, to prove $(1)$ and $(2)$ it suffices to show that any compact complex surface has a model $(\Mm,d)$ with a weight decomposition such that $H^1(\Mm)$ has only weights 1 and 2 and $H^2(\Mm)$ has only weights 2 and 3.
The case $b^1=1$ is trivial, so let us assume that $b^1>1$. In this case, we take the model given in Theorem \ref{theomodelreal}
\[\Mm_\RR=H^{\leq 1}_{BC}\otimes_d\Lambda(\gamma,\gamma^c),\, d\gamma=0,\,d\gamma^c=[\beta]_{BC}.\]
To endow it with a weight decomposition it suffices to declare $H_{BC}^k$ to be of weight k and $\gamma$, $\gamma^c$ to be of weight 2.
Since $[\beta]_{BC}$ has weight 2, the differential on $\Mm_\RR$ is compatible with weights.
Moreover, the product $\gamma\cdot\gamma^c$ is never closed and so the weights in cohomology satisfy the required properties.
For $(3)$, we will just compute the real Malcev completion of $\pi_1(\langle \Mm_\RR\rangle)$.

The algebra $\Mm_\RR$ is a quadratic dg-algebra with generating vector space \[V=\langle u_i,v_i\rangle_{i=1}^k\oplus \langle\gamma,\gamma^c\rangle\] and relations $R$ given by
\[u_i\cdot u_j=0,\, v_i\cdot v_j=0,\, u_i\cdot v_i=0\text{ for all } i\neq j,\, \text{ and }u_i\cdot v_i=u_1\cdot v_1\text{ for all } 2\leq i\leq k.\]
The differential on generators is $d\gamma^c=u_1\cdot v_1$ and trivial otherwise.
To check that it is Koszul, we use the rewriting method (see Theorem 4.1.1 of \cite{LoVa}).
This ensures $\Mm_\RR$ is Koszul if
 $V$ admits an ordered basis, for which there exists a suitable
order on the set of tuples, such that every critical monomial is confluent.
In our case, it suffices to observe that the only (interesting) critical monomials with respect to the above relations are the monomials of the form $u_i\cdot {v_i}\cdot u_j$ or $u_i\cdot v_i\cdot v_j$ and those are obviously confluent.
Now we can proceed
analogously to the main theorem of \cite{Bez}. We use the notation from \cite[Section 3]{Bez}.
Denote by $\{X_i,Y_i\}_{i=1}^k\sqcup\{Z,W\}$ the dual basis of $V^*$. The space $I\subset V\otimes V$ is generated by the relations above. The space $I^{\perp}\subset V^*\otimes V^*$  is generated by all elements of the form $Z\otimes W$, $Z\otimes X_i$, $Z\otimes Y_i$, $W\otimes X_i$ and $W\otimes Y_i$ as well as the element
\[X_1\otimes Y_1-\sum_{i=2}^n X_i\otimes Y_i\]
The map
\[\varphi:I^{\perp}\to V^*\]
is dual to the differential $V\to V\otimes V/I$ so it is zero on every basis element except
\[\varphi(X_1\otimes Y_1-\sum_{i=2}^n X_i\otimes Y_i)=W.\]
This gives the following presentation for the quadratic linear dual of $\Mm_\RR$. It has generators
$\{X_i,Y_i,Z,W\}$ with $1\leq i\leq n$, subject to the following relations:
\begin{enumerate}[(a)]
\item $[Z,W]=0$.
\item $[Z,X_i]=[Z,Y_i]=0$ for all $1\leq i\leq n$.
\item $[W,X_i]=[W,Y_i]=0$, for all $1\leq i\leq n$.
\item $W=[X_1,Y_1]-\sum_{i=2}^n [X_i,Y_i]$.
\end{enumerate}
Substituting $(d)$ into the previous equations, we find that $(a)$ is implied by $(b)$ together with the Jacobi identity,
while $(c)$ gives the desired cubic relations, after a change of basis $X_1\mapsto -X_1$.
\end{proof}

\bibliographystyle{alpha}
\bibliography{biblio}

\end{document}